\documentclass[letterpaper,11pt]{amsart}

\usepackage{amsmath,amssymb}

\usepackage{txfonts}
\usepackage[top=1in, bottom=1in, left=1in, right=1in]{geometry}

\newtheorem{theorem}{Theorem}
\newtheorem{lemma}{Lemma}

\newtheorem{remark}{Remark}

\numberwithin{equation}{section}
\numberwithin{lemma}{section}
\numberwithin{proposition}{section}
\numberwithin{corollary}{section}

\title[Limit of Fractional Power Sobolev Inequalities]{Limit of Fractional Power Sobolev Inequalities}
\author{Sun-Yung Alice Chang}
\address{Princeton University}
\email{chang@math.princeton.edu}
\author{Fang Wang}
\address{Shanghai Jiao Tong University}
\email{fangwang1984@sjtu.edu.cn}

\thanks{Chang's research was supported in part by the NSF grant No. DMS-1509505; 
 Wang's research was supported in part by National Natural Science Foundation of China No. 11401377. }
\begin{document}

\begin{abstract}
We derive the Moser-Trudinger-Onofri inequalities on the 2-sphere and the 4-sphere as the limiting cases of the fractional power Sobolev inequalities on the same spaces, and justify our approach as the dimensional continuation 
argument initiated by Thomas P. Branson.
\end{abstract}

\maketitle

\section{Introduction}
In this paper, we will establish, in two cases, the log exponential 
type Sobolev inequalities as the limiting case of the fractional power 
Sobolev inequalities. One such case is the Moser-Trudinger-Onofri inequality 
on the 2-sphere, which we will establish as the limiting of the fractional 
power Sobolev inequalities on the same 2-sphere. 

Our approach is motivated by a dimensional continuation argument by T. Branson. In \cite{Br1} and in different lectures, Branson often mentioned his way to "guess" the correct formula or inequality on a borderline case through 
a dimension continuation argument. We now illustrate one example of 
his approach. One looks at the classical Sobolev embedding theorem for $W^{1,2}\hookrightarrow L^{\frac{2n}{n-2}}$ on the standard sphere $(\mathbb{S}^n, d\theta^2)$ when $ n \geq 3$. In this case, the embedding can be expressed as the following sharp inequality: 
\begin{equation}\label{eq.1.1}
c_n \left(\fint_{\mathbb{S}^n} |f|^{\frac{2n}{n-2}} d\theta\right)^{\frac{n-2}{n}}  \leq \fint_{\mathbb{S}^n} |\nabla_{\theta}f|^2 d\theta+ c_n \fint_{\mathbb{S}^n}|f|^2 d\theta,
\quad \forall f\in W^{1,2}(\mathbb{S}^n), 
\end{equation}
where 
$$
\fint_{\mathbb{S}^n}  d\theta=\frac{\int_{\mathbb{S}^n} d\theta}{|\mathbb{S}^n|}, \quad
c_n=\frac{n(n-2)}{4}.
$$
On the other hand, when $n =2$, on $\mathbb{S}^2$, there is the famous Moser-Trudinger-Onofri inequality: 
\begin{equation}\label{eq.1.2}
\ln \fint_{\mathbb{S}^2} e^{2\omega} d\theta \leq \fint_{\mathbb{S}^2} |\nabla_{\theta} \omega |^2 d\theta + 2\fint_{\mathbb{S}^2} \omega d\theta,
\quad \forall \omega\in W^{1,2}(\mathbb{S}^2). 
\end{equation}
 See \cite{Mo1}, \cite{On1}.

Traditionally, Moser-Trudinger-Onofri inequality (\ref{eq.1.2}) was derived from a method very different than the derivation of Sobolev inequality (\ref{eq.1.1}). We now describe Branson's heuristic argument to derive (\ref{eq.1.2}) from (\ref{eq.1.1}) by a dimension continuation argument when $n\rightarrow 2$.
To do so, for a given function $f\in W^{1,2}(\mathbb{S}^n)$ when $n\geq 3$,
we may assume it is nonnegative and write $f=e^{\frac{n-2}{2}\omega}$. Then we 
rewrite (\ref{eq.1.1}) as
\begin{equation} \label{eq.1.3}
c_n\left[
\left( \fint_{\mathbb{S}^n} e^{n\omega } d\theta \right)^{\frac{n-2}{n}} -\fint e^{(n-2)\omega} d\theta 
\right]
\leq 
\left(\frac{n-2}{2}\right)^2 \fint_{\mathbb{S}^n} e^{(n-2)\omega} |\nabla_{\theta} \omega|^2 d\theta. 
\end{equation}
Divide both side of (\ref{eq.1.3}) by $(\frac{n-2}{2})^2$ and rewrite the left hand side, we get
\begin{equation}\label{eq.1.4}
\frac{n}{n-2}\left[
\left(\fint_{\mathbb{S}^n} e^{n\omega} d\theta \right)^{\frac{n-2}{n}}-1 -\fint_{\mathbb{S}^n} \left(e^{(n-2) \omega}-1\right) d\theta
\right]
\leq \fint_{\mathbb{S}^n} e^{(n-2)\omega} |\nabla_{\theta} \omega|^2 d\theta. 
\end{equation}
We now observe that if we "formally" let $n\rightarrow 2$ and apply L'H\^{o}pital rule, we get the Moser-Trudinger-Onofri inequality (\ref{eq.1.2}).

In this note we will replace the parameter $n$ in the argument above by a "continuous" parameter, apply some results developed in scattering theory and make above dimension continuation argument by Branson rigorous. More specifically, on the 2-sphere, we consider the fractional power operator $P_{2\gamma} $ (as defined in \S 2.1 below), with leading symbol $|\xi|^{2\gamma}$ (\cite{JS1}), and the corresponding Sobolev embedding inequality for all $0<\gamma<1$ (\cite{Be1}):
\begin{equation}\label{eq.1.5}
Y_{\gamma}(\mathbb{S}^2) \|f\|^2_{L^{\frac{2}{1-\gamma}}(\mathbb{S}^2)} \leq \fint_{\mathbb{S}^2} fP_{2\gamma}f d\theta, 
\quad \forall f\in W^{\gamma,2}(\mathbb{S}^2),
\end{equation}
where $Y_{\gamma} (\mathbb{S}^2)=\frac{\Gamma(1+\gamma)}{\Gamma(1-\gamma)} |\mathbb{S}^2|^{\gamma}$. 

Our main result (Theorem 1 in Section 3) is that we can derive Moser-Trudinger-Onofri inequality (\ref{eq.1.2}) as a limit of the fractional Sobolev inequality (\ref{eq.1.5}) when $\gamma\rightarrow 1$. 

The main tool we will use comes from scattering theory consideration. That is, 
we will apply an extension theorem to view any function $f$ defined on $\mathbb{S}^2$ as the boundary data of a solution $U_f$ of the Poisson equation defined on the 3-ball $\mathbb{B}^3$; and view the function $P_{2\gamma}f $ as the scattering matrix operating on the function $f$. In this sense, the extension function $U_f$ can be viewed as a function on $\mathbb{B}^3$ with weight 
$\rho^{1-2\gamma}$ (where $\rho$ is some distance function from $\mathbb{S}^2$ to $\mathbb{B}^3$): there we can view $U_f$ as defined on a space of dimension $n_{\gamma}=3+(1-2\gamma)$. In this sense, when $\gamma\rightarrow 1$, $n_{\gamma}\rightarrow 2$; thus the fact (\ref{eq.1.5}) tends to Moser-Trudinger-Onofri inequality when $ \gamma$ tends to one can be viewed as a dimensional continuation argument of T. Branson. 

It turns out above argument can be extended to operators of order higher than two. A second main result in the paper is to show the same method can be applied to derive the corresponding Moser-Trudinger-Onofri inequality on 4-sphere (Theorem 2 in Section 5 ) which corresponds to a sharp inequality with respect to 4-th order Paneitz operator on 4-sphere as the limiting case of the corresponding fractional power Sobolev inequalities, like that of (\ref{eq.1.5}), with $1<\gamma<2$ defined on 4-sphere. The proof of this higher order case 
turns out to be technically harder, mainly because we need to choose some suitable distance function $\rho$ (which depends on $\gamma$) with sharp asymptotic
estimates when the distance approaches the boundary.

This paper is organised as follows. In Section 2, we will first recall the 
definition of fractional operators $P_{2 \gamma}$ on the sphere. It turns out
this class of operators is a special case of a general class of fractional order GJMS operator defined on the boundary of conformal compact Einstein manifolds
(e.g. the sphere as boundary of the hyperbolic space) as introduced in
\cite{GJMS1} and \cite{GZ1}. As we have mentioned, the main tool we will use in the proof of our theorems is some Sobolev trace type inequalities --which we will call extension theorems. These extension theorems were first introduced by
 Caffarelli and Silv1estre \cite{CS1} for functions defined
on $\mathbb{R}^n$ with extensions on the upper half space $\mathbb{R}^{n+1}_{+}$and later generalized to functions defined on the boundary of conformal compact Einstein manifolds (\cite{CG1}, \cite{CC1}). We will briefly recall some basic definitions and the statements of these extension theorems in section 2. We remark
that for readers only
interest in the method to prove Theorem 1, one can skip read section 2.2.2, 
the result there is only needed in the proof of Theorem 2.
In Section 3, we derive Moser-Trudinger-Onofri inequality as a limiting case of the fractional Sobolev inequality on 2-sphere. In Section 4, we derive some
further estimates for the special distance function which were called as  
the "adapted geodesic defining" function (\cite{CC1}) and the corresponding weighted scalar curvature and Schouten tensors for adapted metrics on hyperbolic space $\mathbb{H}^5$. Based on these estimates, we derive in section 5 a generalized 
Moser-Trudinger-Onofri type inequality for functions defined on the 4-sphere, which were established earlier by \cite{BCY1} and \cite{Be1},  again as the limit of the fractional Sobolev inequality.

\vspace{0.2in}
\section{Geometric Preliminary}

\subsection{$P_{2 \gamma}$  Operators}

We first recall (\cite{Br1}, \cite{Be1}) that on the standard sphere $(\mathbb{S}^n, [d\theta^2])$, there is a class of pseudo-differential operators 
$P_{2 \gamma} $, defined for $\gamma \in (0, \frac{n}{2})$ as follows:  
$$
\begin{aligned} \label{pgamma}
&P_{2\gamma}=\frac{\Gamma(B+\frac{1}{2}+\gamma)}{\Gamma(B+\frac{1}{2}-\gamma)}, \quad\mathrm{where}\quad 
B=\sqrt{\triangle_{\theta}+\left(\tfrac{n-1}{2}\right)^2}.
\end{aligned}
$$
Associated with the operator, the fractional Q-curvature is defined as:
$$
Q_{2\gamma}=\frac{2}{n-2\gamma} P_{2 \gamma}(1) =  \frac{2}{n-2\gamma} \frac{\Gamma(\frac{n}{2}+\gamma)}{\Gamma(\frac{n}{2}-\gamma)}. 
$$
$P_{2\gamma} $ is an operator with leading symbol $|\xi|^{2\gamma}$ satisfies
some conformal covariant property, for example when $ \gamma =1$, $P_2 = 
\triangle_{\theta}+ \frac {n(n-2)}{4} $ is the conformal Laplace opertor on
$S^n$, $P_4$ is the 4-th order Paneitz operator (\cite{Pa1}).
The property of $P_{2 \gamma}$ which is relevant to us is the following
sharp Sobolev inequality for $\gamma\in (0, \frac{n}{2})$, $q=\frac{2n}{n-2\gamma}$, 
\begin{equation}\label{eq.2.10}
Y_{\gamma}(\mathbb{S}^n) \|f\|^2_{L^{q}(\mathbb{S}^n)}
\leq \int_{\mathbb{S}^n} fP^{}_{2\gamma} fd\theta, \quad \forall \ f\in H^{\gamma}(\mathbb{S}^n).
\end{equation}
See \cite{GQ1} \cite{Be1} \cite{CT1} \cite{CT2} where 
\begin{equation}\label{eq.2.9}
Y_{\gamma}(\mathbb{S}^n) =2^{\frac{2\gamma}{n}}\pi^{\frac{\gamma(n+1)}{n}} \frac{\Gamma(\frac{n}{2}+\gamma)}{\Gamma(\frac{n}{2}-\gamma)}\left[ \Gamma\left(\frac{n+1}{2}\right)\right]^{-\frac{2\gamma}{n}} =\frac{\Gamma(\frac{n}{2}+\gamma)}{\Gamma(\frac{n}{2}-\gamma)} \left|\mathbb{S}^n\right|^{\frac{2\gamma}{n}}. 
\end{equation}
It turns out $P_{2 \gamma} $ on $\mathbb{S}^n$ is a special case of a general
class of fractional GJMS opertors defined on the boundary (called conformal
infinity) of some general conformal compact Einstein manifolds. Here we 
will recall some basic background.

Fractional GJMS Operators are defined on any closed manifold $M$ with a conformal class of metric $[h]$ such that $(M, [h])$ can be embedded as the conformal infinity of a Poincar\'{e}-Einstein manifold $(X^{n+1}, g_+)$. More explicitly, $(X^{n+1}, g_+)$ and $(M, [h])$ should satisfy
\begin{equation*}\label{eintein}
\begin{cases}
Ric_{g_+}=-ng_+ &\textrm{in}\  X, \\
\rho^2g_+|_{TM}\in [h] &\textrm{on}\ M, 
\end{cases}
\end{equation*}
where $\rho$ is a boundary defining function for  $\partial X=M$.

A special example is the Hyperbolic space $\mathbb{H}^{n+1}$ in the ball model:
\begin{equation}\label{ballmodel}
\mathbb{B}^{n+1}=\{x\in\mathbb{R}^{n+1}: |x|<1\}\quad \mathrm{and}\quad g_+ =\frac{4dx^2}{(1-|x|^2)^2}=\frac{4(dr^2+r^2d\theta^2)}{(1-r^2)^2}, 
\end{equation}
where $(r,\theta)$ are the polar coordinates on the unit ball $\mathbb{B}^{n+1}$. Take $h=d\theta^2$, the canonical spherical metric. Then the geodesic normal defining function  w.r.t. $h$ is $\rho=\frac{2(1-r)}{1+r}$. For $\rho\in (0,2)$, 
\begin{equation}\label{hypmetric}
g_+=\rho^{-2}\left(d\rho^2+\left(1-\frac{\rho^2}{4}\right)^2 d\theta^2 \right), \quad \rho^2g_+|_{T\mathbb{S}^n} = d\theta^2. 
\end{equation}

Based on the study of boundary regularity in \cite{CDLS} 
and the spectral and resolvent theorems for Laplacian of $g_+$  in \cite{MM1} \cite{Ma1} \cite{Gu1}, we can consider the following equation:
$$
(\triangle_+ - s(n-s))u=0.
$$
If $\mathrm{Re}(s)>\frac{n}{2},  s(n-s)\notin \sigma_{pp}(\triangle_+),  2s-n\notin \mathbb{N}$, then given any $f\in C^{\infty}(M)$ there is a unique solution satisfying $\rho^{s-n}u|_{M}=f\in C^{\infty}(M)$. Moreover, $u$ takes the form
$$
u=\rho^{n-s} F + \rho^sG, \quad F, G\in C^{\infty}(\overline{X}), \quad F|_{M}=f. 
$$
We define the scattering operator $S(s)$ by
$$
S(s): C^{\infty}(M)\longrightarrow C^{\infty}(M), \quad S(s)f=G|_{M}. 
$$
Here $S(s)$ is a one parameter family of conformally invariant elliptic pseudo-differential operators or order $2s-n$, which can be meromorphically extended to $\mathbb{C}\backslash \{\frac{n-1}{2}-K-\mathbb{N}_0\}$ where $2K$ is the order up to which $g_+$ is even in its boundary asymptotic expansion. 
If $s_0>\frac{n}{2}$ is a pole satisfying $ 2s_0-n\in\mathbb{N},s_0(n-s_0)\notin\sigma_{pp}(\triangle_+)$, then the order of this pole is at most $1$ and the residue is a differential operator on $M$. In particular, if  $\frac{n^2}{4}-k^2\notin\sigma_{pp}(\triangle_+)$ for $k\leq\min\{\frac{n}{2}, K\}$, then
$$
\mathrm{Res}_{s=s_0} S(s)=c_kP_{2k}, \quad c_k=\frac{(-1)^{k-1}}{2^{2k}k!(k-1)!}, 
$$
where $P_{2k}$ is the GJMS operator of order $2k$ on $(M,[h])$. See \cite{JS1}, \cite{GZ1} for more details. 

For simplicity, we define the renormalised scattering operators and the associated curvatures by
$$
\begin{gathered}
P_{2\gamma} =d_{\gamma}S\left(\frac{n}{2}+\gamma\right), \quad
Q_{2\gamma}=\frac{2}{n-2\gamma} P_{2\gamma}(1); \quad
d_{\gamma}=2^{2\gamma}\frac{\Gamma(\gamma)}{\Gamma(-\gamma)}.
\end{gathered}
$$
While $\gamma\notin \mathbb{N}$, $P_{2\gamma}$ is also called the \textit{fractional GJMS operators} and $Q_{2\gamma}$ is the fractional Q-curvature.
From the definition we see the fractional order GJMS operators $P_{2\gamma}$ should also depend on the interior metric $(X^{n+1}, g_+)$, not only on $(M, [h])$. 
A special case is the Hyperbolic space $\mathbb{H}^{n+1}$, which has conformal infinity $(\mathbb{S}^n, [d\theta^2])$.
In this case, the rigidity theorems given in  \cite{ST1} \cite{DJ1} and \cite{LQS1} tell us if $(X^{n+1}, g_+)$ is Poincar\'{e}-Einstein with conformal infinity $(\mathbb{S}^n, [d\theta^2])$, then $(X^{n+1}, g_+)$ must be the Hyperbolic space $\mathbb{H}^{n+1}$. 
So the fractional GJMS operator on the standard sphere is uniquely defined.
\subsection{Extension Formulas.}
The fractional GJMS operator was intensively studied by Case-Chang in \cite{CC1}. Here we recall some extension formulas given there. 
Let $(X^{n+1}, g_+)$ be a Poinar\'{e}-Einstein manifold with conformal infinity $(M, [h])$. We fix a representative $h$ here. 

Let $\gamma\in(0, \frac{n}{2})$ be such that $\mathrm{Spec}(\triangle_+)>\frac{n^2}{4}-\gamma^2$ and $\rho_*$ be the \textit{adapted boundary defining function} defined as follows:
\begin{equation}\label{def.adapted}
\begin{gathered}
\rho_*=v_*^{\frac{1}{n-s}}, \quad\mathrm{where}
\\
\triangle_+ v_*-s(n-s)v_*=0,\quad  \rho^{s-n} v_*|_{M}=1. 
\end{gathered}
\end{equation}
Notice that we always write $s=\frac{n}{2}+\gamma$. 
\subsubsection{$\gamma\in(0,1)$} Set $m_0=1-2\gamma\in(-1,1)$. Then for each $f\in C^{\infty}(M)$, (see \cite{CC1}, Lemma 7.2, equation appeared in the proof of Theorem 7.3)  
\begin{equation}\label{eq.2.1}
\int_M fP_{2\gamma}f \mathrm{dvol}_h -\frac{n-2\gamma}{2}\int_M Q_{2\gamma} |f|^2 \mathrm{dvol}_h
= -\frac{d\gamma}{2\gamma}\int_X |\nabla U_f|^2 \rho_*^{m_0} d\mathrm{dvol}_{g_*}. 
\end{equation}
Here $g_*=\rho_*^2g_+$,  and $U_f$ is uniquely determined by the equation
\begin{equation}\label{eq.2.2}
\begin{cases}
\triangle_{\phi_0} U=0, &\ \mathrm{in}\ X^{n+1}, 
\\
U=f, &\ \mathrm{on}\ M, 
\end{cases}
\end{equation}
where $\triangle_{\phi_0}= \triangle_{g_*} -m_0 \rho_*^{-1}\nabla \rho_*$ is the weighed Laplacian and $\nabla$ is the connection w.r.t. $g_*$. The equation (\ref{eq.2.2}) is equivalent to the following one:
\begin{equation}\label{eq.2.3}
\triangle_+ u -s(n-s)u=0, \quad \rho_*^{s-n}u|_{M}=f,
\end{equation}
via the transformation $u=\rho_*^{n-s}U$. 
\begin{lemma}\label{lem.2.1}
Let $\gamma\in(0,1)$ and $m_0=1-2\gamma$.
 If the boundary dimension $n\geq 2$ and the Yamabe invariant of $(M,[h])$ is nonnegative, then for arbitrary $f\in C^{\infty}(M)$, the function $U_f$ defined in (\ref{eq.2.2}) is the unique minimizer of  the functional 
$$
I_0(V)=\int_{X^{n+1}} |\nabla V|^2 \rho_*^{m_0} \mathrm{dvol}_{g_*}, \quad V\in\mathcal{V}^0_f.
$$
Here $\mathcal{V}^0_f$ is the function space
$$
\mathcal{V}^0_f=\left\{ V\in C^{l, \alpha}(\overline{X}^{n+1})\cap C^2(X^{n+1}): V|_{M}=f
\right\}.
$$
If $\gamma\in (0,1/2)$, then $(l,\alpha)=(0,2\gamma)$; if $\gamma\in[1/2,1)$, then $(l,\alpha)=(1,2\gamma-1)$. 
\end{lemma}
\begin{proof}
Notice that  $U$ is a minimizer of $I_0(V)$  if and only if  $U$ satisfies the equation (\ref{eq.2.2}), or equivalently, $u=\rho_*^{n-s}U$ satisfies (\ref{eq.2.3}). The uniqueness comes from the uniqueness of solution to (\ref{eq.2.3}). Since the Yamabe invariant of the conformal infinity is nonnegative, there is no $L^2$-eigenvalue for $\triangle_+$. See \cite{Le1}. Hence  (\ref{eq.2.3}) has a unique solution. 
\end{proof}

\subsubsection{$\gamma\in(1,2)$} Set $m_1=3-2\gamma\in (-1,1)$. Then for each $f\in C^{\infty}(M)$, (see \cite{CC1}, Lemma 7.6, equation (7.17)) 
\begin{equation}\label{eq.2.4}
\begin{aligned}
&\int_M fP_{2\gamma}f \mathrm{dvol}_h -\frac{n-2\gamma}{2}\int_M Q_{2\gamma} |f|^2 \mathrm{dvol}_h
\\
=&\ 
 \frac{d_{\gamma}}{8\gamma(\gamma-1)} 
\int_{X}  \left(  |\triangle_{\phi_1} U_f |^2+(n+m_1-1)J_{\phi_1}^{m_1}|\nabla U_f|^2_{g_*}
-4P_{ij}\nabla^iU_f \nabla^jU_f \right) \rho_*^{m_1} \mathrm{dvol}_{g_*}.
\end{aligned}
\end{equation}
Here $U_f$ is uniquely determined by the equation
\begin{equation}\label{eq.2.5}
\begin{cases}
L_{4,\phi_1}^{m_1} U=0, & \textrm{in $X$,}
\\
U=f, & \textrm{on $M$,}
\\
\lim_{\rho_* \rightarrow 0} \rho_*^{m_1} \frac{\partial U}{\partial \rho_*}=0. &
\end{cases}
\end{equation}
Here $L_{4,\phi_1}^{m_1}$ is the weighted GJMS operator of 4-th order w.r.t. $g_*$; $ J_{\phi_1}^{m_1} $ and $P_{\phi_1}^{m_1}=P_{g_*}$ are the weighted scalar curvature and weighted Schouten tensor of $g_*$.
Please refer to \cite{CC1} for explicit definitions. 
By conformal transformation $u=\rho_*^{n-s}U$, equation (\ref{eq.2.5}) is equivalent to
\begin{equation}\label{eq.2.6}
\begin{gathered}
\left[\triangle_+-(s-2)(n-s+2) \right]  \left[\triangle_+ -s(n-s)\right]u=0, 
\\
\quad \rho_*^{s-n}u|_{M}=f,\quad \rho_*^{m_1}\partial_{\rho_*}(\rho_*^{s-n}u)|_{M}=0. 
\end{gathered}
\end{equation}
See formulae (3.3) in \cite{CC1}.

\begin{lemma}\label{lem.2.2}
Let $\gamma\in(1,2)$ and $m_1=3-2\gamma$. 
If the boundary dimension $n\geq 4$ and  the Yamabe invariant of $(M,[h])$ is nonnegative, then for arbitrary $f\in C^{\infty}(M)$, the function $U_f$ defined in (\ref{eq.2.5}) is the unique minimizer of  the functional 
$$
I_1(V)=\int_{X}  \left(  |\triangle_{\phi_1} U_f |^2+(n+m_1-1)J_{\phi_1}^{m_1}|\nabla U_f|^2_{g_*}
-4P_{ij}\nabla^iU_f \nabla^j U_f \right) \rho_*^{m_1} \mathrm{dvol}_{g_*}.
$$
Here $\mathcal{V}^1_f$ is the function space
$$
\mathcal{V}^1_f=\left\{ V\in C^{l, \alpha}(\overline{X}^{n+1})\cap C^4(X^{n+1}): V|_{M}=f, \lim_{\rho_* \rightarrow 0} \rho_*^{m_1} \frac{\partial V}{\partial \rho_*}=0
\right\}.
$$
If $\gamma\in (0,3/2)$, then $(l,\alpha)=(2,2\gamma-2)$; if $\gamma\in[3/2,2)$, then $(l,\alpha)=(3,2\gamma-3)$. 
\end{lemma}
\begin{proof}
According to \cite{CC1}, the functional $I_1(V)$ is nonnegative for $n\geq 4$. The minimizer is a solution to (\ref{eq.2.5}). When the boundary Yamabe type is nonnegative, the equation (\ref{eq.2.5}), or equivalently (\ref{eq.2.6}), has a unique solution. This is because the second boundary condition implies that 
$$
\triangle_+ u-s(n-s)u=0, \quad \rho_*^{s-n}u|_{M}=f.
$$
It is obviously that this has a unique solution. 
\end{proof}

\subsubsection{Some remarks on the adapted defining function}
From the definition, $\rho_*$ depends on $\gamma$. So while $\gamma$ is not fixed, we need some uniform control of this function. Here we give a global uniform value estimate of $\rho_*$ by a special boundary defining function $\rho_L$. 
Define
\begin{equation}\label{eq.2.7}
\begin{gathered}
\rho_L=v_L^{-1},\quad \mathrm{where}\quad
\\
(\triangle_+ +n+1)v_L=1, \quad \rho v_L|_{M}=1.
\end{gathered}
\end{equation}
Recall that $\rho$ is the geodesic normal defining function corresponding to boundary metic $h$.
This $v_L$ was first introduced by Lee in \cite{Le1}  to construct $L^2$-test functions. Lee's computation  implies the following Lemma:
\begin{lemma}\label{lem.2.3}
If $(M, [h])$ is of nonnegative Yamabe type, then for $s\in (0,n)$ and $\rho_L$ defined in (\ref{eq.2.7}), the function $\psi=\rho_L^{n-s}$ is positive and satisfies
$$
\left(\triangle_+ -s(n-s)\right) \psi>0. 
$$
\end{lemma}
\begin{lemma}\label{lem.2.4}
If $(M, [h])$ is of nonnegative Yamabe type, then for $s\in (0,n)$, the adapted boundary defining function $\rho_*$ satisfies
$$
0<\frac{\rho_*}{\rho_L}\leq 1. 
$$
\end{lemma}
\begin{proof} 
For $(M, [h])$ of nonnegative Yamabe type, $\mathrm{Spec}(\triangle_+)>\frac{n^2}{4}$. So for all $s\in (0,n)$, the adapted defining function is well defined, i.e. $\rho_*=v_*^{\frac{1}{n-s}}$, where
$$
\left(\triangle_{+}-s(n-s)\right)v_*=0, \quad \rho^{s-n}v_*|_{M}=1. 
$$
Notice that if taking $\psi=\rho_L^{n-s}$, then
$$
\left(\frac{\rho_*}{\rho_L}\right)^{n-s} =\frac{v_*}{\psi},\quad
 \left(\frac{v_*}{\psi}\right)\vline_{M}=1. 
$$
Direct computation shows that $v_*/\psi$ satisfies
$$
\triangle_+\left(\frac{v_*}{\psi}\right)=\left(s(n-s)- \frac{\triangle_+\psi}{\psi}\right)\left(\frac{v_*}{\psi}\right)+2\nabla \left(\frac{v_*}{\psi}\right) \frac{\nabla\psi}{\psi}.
$$
Applying Lemma \ref{lem.2.3} and the maximum principle to above equation, we get 
$$
0<\frac{v_*}{\psi}\leq 1\quad\Longrightarrow\quad  0<\frac{\rho_*}{\rho_L}\leq 1. 
$$
We finish the proof.
\end{proof}
Refined value and derivative estimates for $\rho_*$ on hyperbolic space will be given in Section 4 before we get to the continuation argument on 4-sphere.

\vspace{0.2in}
\section{Inequalities on $\mathbb{S}^2$}

In this section, we will establish the following Theorem: 
\begin{theorem}\label{thm.1}
We can derive the Moser-Trudinger-Onofri inequality (\ref{eq.1.2}) as the limit of 
the Sobolev inequality (\ref{eq.1.5}).
\end{theorem}
\begin{proof}  
To do so, for a given function $\omega\in C^{\infty}(\mathbb{S}^2)$,  
choose $f=e^{(1-\gamma)\omega}$,  apply the inequality (\ref{eq.1.5}) to $f$ and rewrite
the inequality as

\begin{equation}\label{eq.4.2}
Y_{\gamma}(\mathbb{S}^2) \left(\int_{\mathbb{S}^2}|f|^\frac{2}{1-\gamma} d\theta \right)^{1-\gamma}
 - (1-\gamma)Q_{2\gamma} \int_{\mathbb{S}^2} |f|^2 d\theta
\leq\int_{\mathbb{S}^2} P_{2\gamma}f   d\theta- (1-\gamma)Q_{2\gamma} \int_{\mathbb{S}^2} |f|^2 d\theta. 
\end{equation}
Divide both side of (\ref{eq.4.2}) by $(1-\gamma)^2$, and let $\gamma$ tends to $1$.  We will show that
the left hand side of (\ref{eq.4.2}) tends to $ 4 \pi$ times $\ln \fint_{\mathbb{S}^2} e^{2(\omega- \bar{\omega})} d\theta$, where $ \bar{\omega} $ is the average of $\omega$ over $\mathbb{S}^2$,  and the right side of (\ref{eq.4.2}) tends to $ 4 \pi$ times 
the energy term $\fint_{\mathbb{S}^2} |\nabla_{\theta} \omega |^2 d\theta$ , 
which establishes the theorem.

To prove so, we first recall that 
\begin{equation}\label{eq.4.1}
Y_{\gamma}(\mathbb{S}^2) =(1-\gamma)Q_{2\gamma} |\mathbb{S}^2|^{\gamma}=\frac{\Gamma(1+\gamma)}{\Gamma(1-\gamma)} (4\pi)^{\gamma}. 
\end{equation} 
Denote $m_0=1-2\gamma\in (-1,1)$ and $g_*=\rho_*^2g_+$, where $\rho_*$ is the adapted boundary defining function. 
Let $U_f$ be defined by equation (\ref{eq.2.2}). 
Then  (\ref{eq.2.1}) (\ref{eq.4.2}) and Lemma \ref{lem.2.1} give the following:
\begin{equation}\label{eq.4.3}
\begin{aligned}
&\ \frac{\Gamma(1+\gamma)}{\Gamma(1-\gamma)}
\left[
(4\pi)^{\gamma} \left(\int_{\mathbb{S}^2}|f|^\frac{2}{1-\gamma} d\theta \right)^{1-\gamma} -\int_{\mathbb{S}^2} |f|^2 d\theta
\right]
\\
\leq &\ 
\int_{\mathbb{S}^2} \left(fP_{\gamma}f -(1-\gamma)Q_{\gamma}|f|^2\right) d\theta
\\
=&\ 
-\frac{d_{\gamma}}{2\gamma}\int_{\mathbb{B}^3} |\nabla U_f|_g^2 \rho_*^{m_0} \mathrm{dvol}_{\rho_*}
\\
\leq &\ 
-\frac{d_{\gamma}}{2\gamma}\int_{\mathbb{B}^3} |\nabla V|_g^2 \rho_*^{m_0} \mathrm{dvol}_{\rho_*}, \quad \forall\  V\in\mathcal{V}^0_f. 
\end{aligned}
\end{equation}
Notice that $(1-\gamma)Q_{\gamma}=\Gamma(1+\gamma)/\Gamma(1-\gamma)$. Here $\nabla$ is the connection w.r.t. metric $g_*$.

Next, we extend $\omega$ to  a smooth function on the ball in the following way:
$$
\Omega(r,\theta)=\chi(r)\omega(\theta),
$$
where $\chi$ is smooth and satisfies $0\leq \chi\leq 1$, $\chi(r)=1$ for $r\in[\frac{2}{3},1]$ and $\chi(r)=0$ for $r\in[0,\frac{1}{3}]$. 
Define
$$
V=e^{(1-\gamma)\Omega}.
$$
So $V|_{r=1}=e^{(1-\gamma)\omega}=f$ and $V\in\mathcal{V}^0_f$ for all $\gamma\in(0,1)$. Then (\ref{eq.4.3}) implies that
\begin{equation}\label{eq.4.4}
\begin{aligned}
&\  \frac{\Gamma(1+\gamma)}{\Gamma(1-\gamma)} \left[
(4\pi)^{\gamma} \left(\int_{\mathbb{S}^2} e^{2\omega} d\theta \right)^{1-\gamma} -\int_{\mathbb{S}^2} e^{2(1-\gamma)\omega} d\theta
\right]
\\
\leq &\ 
 -\frac{d_\gamma}{2\gamma}(1-\gamma)^2 \int_{\mathbb{B}^3} e^{2(1-\gamma)\Omega}|\nabla \Omega|_g^2 \rho_*^{m_0} \mathrm{dvol}_{g_*}. 
\end{aligned}
\end{equation}
For $\gamma\in (0,1)$, divide both side of (\ref{eq.4.4}) by $(1-\gamma)^2$ and we get
\begin{equation}\label{eq.4.5}
\begin{aligned}
A_0(\gamma,\omega)=&\ \frac{\Gamma(1+\gamma)}{\Gamma(2-\gamma)} \frac{1}{(1-\gamma)} \left[
(4\pi)^{\gamma} \left(\int_{\mathbb{S}^2} e^{2\omega} d\theta \right)^{1-\gamma} -\int_{\mathbb{S}^2} e^{2(1-\gamma)\omega} d\theta
\right]
\\
\leq &\
-\frac{d_{\gamma}}{2\gamma}\int_{\mathbb{B}^3} e^{2(1-\gamma)\Omega}|\nabla \Omega|_g^2 \rho_*^{m_0} \mathrm{dvol}_{g_*}
\\ 
= &\ 
 \frac{\Gamma(\gamma)}{\Gamma(2-\gamma)} 2^{2\gamma-1}(1-\gamma)
\int_{\mathbb{B}^3} \left(\frac{\rho_*}{\rho_L}\right)^{m_0+1}  e^{2(1-\gamma)\Omega}|\tilde{\nabla} \Omega|_{g_L}^2 \rho_L^{m_0} \mathrm{dvol}_{g_L}
\\
\leq &\ 
 \frac{\Gamma(\gamma)}{\Gamma(2-\gamma)} 2^{2\gamma-1}(1-\gamma)
\int_{\mathbb{B}^3}  e^{2(1-\gamma)\Omega}|\tilde{\nabla} \Omega|_{g_L}^2 \rho_L^{m_0} \mathrm{dvol}_{g_L}
\\
=&\ B_0(\gamma,\Omega). 
\end{aligned}
\end{equation}
Here  $\rho_L$ is Lee's boundary defining function  defined in (\ref{eq.2.7}) and $g_L=\rho_L^2 g_+$. Then 
$$
g_*=\rho_*^2g_+=\left(\frac{\rho_*}{\rho_L}\right)^{2} g_L, \quad \mathrm{dvol}_{g_*}=\left(\frac{\rho_*}{\rho_L}\right)^{3} \mathrm{dvol}_{g_L}, \quad
|\nabla\Omega|^2_{g_*} = \left(\frac{\rho_*}{\rho_L}\right)^{-2}|\tilde{\nabla} \Omega|_{g_L}. 
$$
The connection $\tilde{\nabla}$ is w.r.t. metric $g_L$.  The last inequality in (\ref{eq.4.5}) is from Lemma \ref{lem.2.4}. Then the theorem is proved by Lemma \ref{lem.4.1} and Lemma \ref{lem.4.2}. 
\end{proof}

\begin{lemma} \label{lem.4.1}
For $\gamma\in(0,1)$ and $A_0(\gamma,\Omega)$ defined in (\ref{eq.4.5}), 
$$
\lim_{\gamma\rightarrow 1-} A_0(\gamma,\omega)=4\pi
\ln \left(\frac{1}{4\pi}\int_{\mathbb{S}^2} e^{2(\omega-\bar{\omega})} d\theta\right). 
$$
\end{lemma}
\begin{proof}
Notices that 
\begin{equation}\label{eq.4.6}
\begin{aligned}
 &\ (4\pi)^{\gamma}  \left(\int_{\mathbb{S}^2} e^{2\omega} d\theta \right)^{1-\gamma} -\int_{\mathbb{S}^2} e^{2(1-\gamma)\omega} d\theta
=
4\pi  \left( \left( \frac{1}{4\pi} \int_{\mathbb{S}^2} e^{2\omega} d\theta \right)^{1-\gamma} -1\right)
-\int_{\mathbb{S}^2}\left( e^{2(1-\gamma)\omega} -1\right) d\theta
\end{aligned}
\end{equation}
Then by L'H\^{o}pital's rule
\begin{equation}\label{eq.4.7}
\begin{aligned}
\lim_{\gamma\rightarrow 1-} A_0(\gamma,\omega) 
=\ & 4\pi  \ln \left(\frac{1}{4\pi}\int_{\mathbb{S}^2} e^{2\omega} d\theta\right) - \int_{\mathbb{S}^2} 2\omega d\theta
= 4\pi
\ln \left(\frac{1}{4\pi}\int_{\mathbb{S}^2} e^{2(\omega-\bar{\omega})} d\theta\right), 
\end{aligned}
\end{equation}
We finish the proof. 
\end{proof}

\begin{lemma} \label{lem.4.2}
For $\gamma\in(0,1)$ and $B_0(\gamma,\Omega)$ defined in (\ref{eq.4.5}), 
$$
\lim_{\gamma \rightarrow 1-} B_0(\gamma,\Omega)=\int_{\mathbb{S}^2} |\nabla_{\theta} \omega|^2 d\theta. 
$$
\end{lemma}
\begin{proof} 
Split the integral in $B_0(\gamma, \Omega)$ into two parts: 
\begin{equation}\label{eq.4.9}
\begin{aligned}
B_0(\gamma, \Omega) = &\ 
\frac{\Gamma(\gamma)}{\Gamma(2-\gamma)} 2^{2\gamma-1}(1-\gamma)\int_{\mathbb{B}^3} e^{2(1-\gamma)\Omega}|\tilde{\nabla} \Omega|_{g_L}^2 \rho_L^{m_0} \mathrm{dvol}_{g_L}
\\
=&\ 
\frac{\Gamma(\gamma)}{\Gamma(2-\gamma)} 2^{2\gamma-1} (1-\gamma)
\int_{\mathbb{S}^2\times [\frac{1}{3}, \frac{2}{3}]} e^{2(1-\gamma)\Omega}|\tilde{\nabla} \Omega|_{g_L}^2 
\rho_L^{1-2\gamma}  \mathrm{dvol}_{g_L}
\\
&\ 
+\frac{\Gamma(\gamma)}{\Gamma(2-\gamma)} 2^{2\gamma-1}(1-\gamma)
\int_{\mathbb{S}^2\times [\frac{2}{3}, 1]} e^{2(1-\gamma)\Omega}|\tilde{\nabla} \Omega|_{g_L}^2  \rho_L^{1-2\gamma}  \mathrm{dvol}_{g_L}
\\
=&\ I(\gamma, \Omega)+II(\gamma, \Omega)
\end{aligned}
\end{equation}
Notice that $\Omega, \rho_L, g_L$ are smooth on the ball which are independent of $\gamma$. So there exists a constant $c>0$ such that 
$$
c^{-1}\leq \frac{\rho_L}{1-r}\leq c. 
$$
Therefore $|\tilde{\nabla}\Omega|_{g_L}^2$ and $e^{2(1-\gamma)\Omega}$ are  uniformly bounded on the ball from above and below. 

For $I(r, \Omega)$, since $r\leq 2/3$,  there exists a constant $\bar{c}>0$ independent of $\gamma\in(0,1)$ such that
$$
\bar{c}^{-1}<\rho_L^{1-2\gamma}<\bar{c},
$$
and hence  a constant $C$ independent of $\gamma\in(0,1)$ such that
$$
\left| \int_{\mathbb{S}^2\times [\frac{1}{3},\frac{2}{3}]} e^{2(1-\gamma)\Omega}|\tilde{\nabla} \Omega|_{g_L}^2
\rho_L^{1-2\gamma}  \mathrm{dvol}_{g_L} \right|
\leq C.
$$
This implies that
\begin{equation}\label{eq.limitI}
\lim_{\gamma\rightarrow 1} I(\gamma,\Omega)=0. 
\end{equation}

For $II(\gamma, \Omega)$, since $2/3< r< 1$, $\Omega(r,\theta)=\omega(\theta)$ and
$$
|\tilde{\nabla} \Omega|_{g_L}^2 =\frac{(1-r^2)^2 }{4\rho_L^2 r^2} |\nabla_{\theta} \omega|^2_{}, 
\quad
 e^{2(1-\gamma)\Omega}=e^{2(1-\gamma)\omega},
  \quad
\mathrm{dvol}_{g_L} = \left(\frac{4\rho_L^2}{(1-r^2)^2}\right)^{\frac{3}{2}} r^2drd\theta. 
$$
So the integral  in $II$ can be rewritten as
\begin{equation}\label{eq.4.10}
II(\gamma, \Omega)
=\frac{\Gamma(\gamma)}{\Gamma(2-\gamma)} 2^{2\gamma-1}(1-\gamma)
\left( \int_{\frac{2}{3}}^1 \frac{2\rho_L^{2(1-\gamma)}}{1-r^2} dr\right)
\left(\int_{\mathbb{S}^2} e^{2(1-\gamma)\omega}|\nabla_{\theta} \omega|^2_{} d\theta\right). 
\end{equation}
Obviously, 
\begin{equation}\label{eq.4.11}
\lim_{\gamma\rightarrow 1-}\int_{\mathbb{S}^2} e^{2(1-\gamma)\omega}|\nabla_{\theta} \omega|^2_{} d\theta
=\int_{\mathbb{S}^2} |\nabla_{\theta} \omega|^2_{} d\theta.
\end{equation}
The integral in variable $r$ can be estimated as follows. First
\begin{equation}\label{eq.4.12}
 \frac{2\rho_L^{2(1-\gamma)}}{1-r^2} = \frac{\rho_L^{2(1-\gamma)}}{1-r} +\frac{\rho_L^{2(1-\gamma)}}{1+r} 
 =\left( \frac{\rho_L}{1-r} \right)^{2(1-\gamma)} (1-r)^{1-2\gamma} +\frac{\rho_L^{2(1-\gamma)}}{1+r} 
\end{equation}
The integration of each term above can be estimated by
\begin{equation}\label{eq.4.13}
\begin{gathered}
  \frac{1}{2(1-\gamma)}\left(\frac{1}{3c}\right)^{2(1-\gamma)} \leq 
\int_{\frac{2}{3}}^1 \left( \frac{\rho_L}{1-r} \right)^{2(1-\gamma)} (1-r)^{1-2\gamma}  dr\leq  \frac{1}{2(1-\gamma)}\left(\frac{c}{3}\right)^{2(1-\gamma)} ;
\\
0< \int_{\frac{2}{3}}^1 \frac{\rho_L^{2(1-\gamma)}}{1+r} dr\leq c^{2(1-\gamma)}\ln\left(\frac{6}{5}\right). 
\end{gathered}
\end{equation}
Take $\gamma\rightarrow 1$, the estimates in (\ref{eq.4.13}) give
\begin{equation}\label{eq.4.14}
\lim_{\gamma\rightarrow 1-} 2(1-\gamma)\left( \int_{\frac{2}{3}}^1 \frac{2\rho_L^{2(1-\gamma)}}{1-r^2} dr\right) =1. 
\end{equation}
Finally, (\ref{eq.4.10}) (\ref{eq.4.11}) and (\ref{eq.4.14}) give
\begin{equation}\label{eq.limitII}
\lim_{\gamma\rightarrow 1-}II(\gamma, \Omega)= \int_{\mathbb{S}^2} |\nabla_{\theta} \omega|^2_{} d\theta. 
\end{equation}
By (\ref{eq.limitI}) and (\ref{eq.limitII}), we finish the proof. 
\end{proof}

\begin{remark}
The proof of the sharp Moser-Trudinger-Onofri inequality above
depends on the delicate choice of the defining function $\rho_{*} $,
other choice of $\rho$ would derive an inequality with an added constant
on the right hand side--which is the original form of the inequality derived by J. Moser. 
\end{remark}

\vspace{0.2in}
\section{Refined Estimates on the Hyperbolic Space}
In this section, we give some refined estimates for the adapted defining function $\rho_*$ and the curvatures for adapted metric $g_*$ on the Hyperbolic space $\mathbb{H}^{n+1}$. Recall the ball model:
$$
\mathbb{B}^{n+1}=\{x\in\mathbb{R}^{n+1}: |x|<1\}, \quad g_+ =\frac{4dx^2}{(1-|x|^2)^2}=\frac{4(dr^2+r^2d\theta^2)}{(1-r^2)^2}.
$$
In this case, we have 4 boundary defining functions: 
\begin{itemize}
\item[(1)] 
The geodesic normal defining function
$$
\rho=\frac{2(1-r)}{1+r},
$$ 
which is not globally smooth  but gives good asymptotic expansion for the hyperbolic metric (\ref{ballmodel}).
\item[(2)]
Lee's defining function
$$
\rho_L =\frac{1-r^2}{1+r^2},
$$
and $g_L= \rho_L^2 g_+$
is the $(n+1)$-spherical metric on the half sphere $\mathbb{S}^{n+1}_+$.
\item[(3)] 
The flat defining function
$$
\rho_0=\frac{1-r^2}{2},
$$
and $g_0=\rho_0^2 g_+$
is the Euclidean metric on the unit disk. 
\item[(4)] 
The adapted geodesic normal defining function $\rho_*$  defined in (\ref{def.adapted}). 
\end{itemize}

\subsection{Uniform bounds and boundary expansion of $\rho_*$}
\begin{lemma}\label{lem.3.1}
On $\mathbb{H}^{n+1}$, 
for $n\geq 2$ and $s=\frac{n}{2}+\gamma\in[\frac{n+1}{2},n)$, 
$$
\rho_0\leq \rho_*\leq \rho_L. 
$$
\end{lemma}
\begin{proof}
By Lemma \ref{lem.2.4}, we only need to show that $\rho_0\leq \rho_*$. 
For $s\in[\frac{n+1}{2},n)$, let $\phi = \rho_0^{n-s}$. Then direct computation shows that 
\begin{equation*}\label{eq.def1}
\triangle_+ \phi - s(n-s)\phi=(n-s)(n+1-2s)\rho_0^{n-s+1} \leq 0.
\end{equation*}
Let $v_* =\rho_*^{n-s}$. Then by (\ref{def.adapted}),
\begin{equation*}\label{eq.def2}
 \triangle_+ v_* - s(n-s)v_*=0. 
\end{equation*}
Consider the equation for $\phi/v_*$:
$$
\triangle_+\left(\frac{\phi}{v_*}\right)=\left( \frac{\triangle_+\phi}{\phi}-s(n-s)\right)\left(\frac{\phi}{v_*}\right)+2\nabla \left(\frac{\phi}{v_*}\right) \frac{\nabla v_*}{v_*}. 
$$
Comparison argument similar as Lemma \ref{lem.2.4} shows that $\phi\leq v_*$ and hence  $\rho_0\leq \rho_*$. 
\end{proof}

\begin{lemma}
On $\mathbb{H}^{n+1}$ , for $s=\frac{n}{2}+\gamma\in(\frac{n}{2}+1, n)$ and $n\geq 4$, 
 the adapted defining function $\rho_*$ has an asymptotic expansion 
$$
\rho_*=\rho\left(  1 -\frac{n}{4(2s-n-2)} \rho^2 +O(\rho^4)\right)+ \rho^{2\gamma+1}\left(\frac{1}{d_{\gamma}} \frac{\Gamma(\frac{n}{2}+\gamma)}{\Gamma(\frac{n}{2}-\gamma)}+O(\rho^2)\right). 
$$
\end{lemma}

\subsection{Derivative estimates for  $\rho_*$}
Since $\rho_0$ is fixed, we give the derivative estimates for  adapted defining function in terms of 
$$
t=\frac{\rho_*}{\rho_0} \quad\mathrm{and}\quad T=\ln t.
$$
Notice that $t=t(r)$ and $T=T(r)$ are both radial symmetric functions. Since $t$ and $T$ are  smooth in the interior, we have
$$
t'(0)=T'(0)=0.
$$
On the boundary, the asymptotic expansion of $\rho_*$ gives the boundary values of $t$ and $T$:
$$
t(1)=1, \quad T(1)=0.
$$
Moreover, if $s=\frac{n}{2}+\gamma\in(\frac{n}{2}+1,n)$ and $n\geq 4$, then
$$
\begin{gathered}
t'(1)=-1, \quad  T'(1)= -1.
\\
t''(1)=\frac{3}{2}-\frac{n}{2(2s-n-2)}, \quad T''(1)=-\frac{n+1-s}{2s-n-2}.
\end{gathered}
$$
To get the derivative estimates for $T$ and hence $t$, we deduce the ODE for $T$ from
$$
\triangle_+ v_*-s(n-s)v_*=0, \quad\mathrm{where}\quad v_*= \rho_*^{n-s} =\rho_0t=\rho_0e^{T}.
$$
Direct computation shows that $T$ satisfies the equation
\begin{equation}\label{eq.T}
T'' + (n-s) (T')^2 + \left(\frac{2s-n-1}{\rho_0}r+\frac{n}{r}\right) T' +\frac{2s-n-1}{\rho_0}=0.
\end{equation}
This is a first order nonlinear ODE for $T'$. 
For simplicity, we denote $F(r)=T'(r)$. Then $F$ satisfies
\begin{equation}\label{eq.F}
\begin{gathered}
F' + (n-s) F^2 + \left(\frac{2s-n-1}{\rho_0}r+\frac{n}{r}\right) F+\frac{2s-n-1}{\rho_0}=0, 
\\
F(0)=0, \quad F(1)=-1. 
\end{gathered}
\end{equation}
\begin{lemma}\label{lem.3.2}
Suppose $s=\frac{n}{2}+\gamma \in (\frac{n}{2}+1,n)$ and $n\geq 4$. Then for all $r\in [0,1]$, 
$$
0< 1+rF\leq 1. 
$$
\end{lemma}
\begin{proof} Here $F$ already exists and is $C^1$ up to boundary for every fixed $s\in  (\frac{n}{2}+1,n)$. Let 
$$
\Phi(r)=\int_1^{r} (n-s)F(\bar{r})d\bar{r}.
$$
First, $2s-n-1>0$ and hence equation (\ref{eq.F}) gives
$$
\left( \frac{r^ne^{\Phi(r)}}{\rho_0^{2s-n-1}} F \right)' = - \frac{(2s-n-1)r^ne^{\Phi(r)}}{\rho_0^{2s-n}} \leq 0.
$$
This shows that $\rho_0^{-(2s-n-1)} r^ne^{\Phi(r)} F $ is decreasing. Since it is valued $0$ at $r=0$, we get  $F(r)\leq 0$ for $r\in [0,1]$. 
Second, denote $\rho_0=\rho_0(r)$. Then
$$
\begin{aligned}
\int_0^r \frac{(2s-n-1)\bar{r}^ne^{\Phi(\bar{r})}}{\rho_0(\bar{r})^{2s-n}} d\bar{r}
&=\int_0^r \bar{r}^{n-1}e^{\Phi(\bar{r})} d \left(\frac{1}{\rho_0(\bar{r})^{2s-n-1}}\right)
\\
&= \frac{r^{n-1}e^{\Phi(r)}}{\rho_0^{2s-n-1}} -\int_0^{r}\frac{[(n-s)\bar{r}F(\bar{r})+n-1]\bar{r}^{n-2}e^{\Phi(\bar{r})}}{\rho_0(\bar{r})^{2s-n-1}} d\bar{r}.
\end{aligned}
$$
This implies that
$$
F(r)=-\frac{1}{r} + \frac{\rho_0^{2s-n-1}}{r^n e^{\Phi(r)}}  \int_0^{r}\frac{[(n-s)\bar{r}F(\bar{r})+(n-1)]\bar{r}^{n-2}e^{\Phi(\bar{r})}}{\rho_0(\bar{r})^{2s-n-1}} d\bar{r}. 
$$
And hence
\begin{equation}\label{eq.integralF}
1+rF(r) =  \frac{\rho_0^{2s-n-1}}{r^{n-1} e^{\Phi(r)}}       \int_0^{r}\frac{[(n-s)(1+\bar{r}F(\bar{r}))+(s-1)]\bar{r}^{n-2}e^{\Phi(\bar{r})}}{\rho_0(\bar{r})^{2s-n-1}} d\bar{r}. 
\end{equation}
At $r=0$, we have $1+rF(r)=1$ for every fixed $s$. Let
$$
R_s=\sup \{R\in[0,1]: 1+rF(r)>0, \forall\ r\in[0,R] \}. 
$$
By continuity  $R_s>0$ for every fixed $s\in  (\frac{n}{2}+1,n)$. If $R_s<1$, then $1+rF(r)|_{r=R_s}=0$ and $1+rF(r)>0$ for all $r\in[0,R_s)$. So the integral in (\ref{eq.integralF})
is strictly positive at $r=R_s$. This gives a contradiction. Hence $R_s=1$ for all $s\in  (\frac{n}{2}+1,n)$ and we finish the proof.
\end{proof}
\begin{lemma}\label{lem.3.3}
For $s=\frac{n}{2}+\gamma \in (\frac{n}{2}+1,n)$ and $n\geq 4$, there exists a constant $C>0$ independent of $s$ and $r$  such that for all $r\in [0,1)$, 
$$
0\leq \frac{1+rF}{\rho_0}\leq \frac{C}{2s-n-2}.
$$
\end{lemma}
\begin{proof}
First notice that from Lemma \ref{lem.3.2} we have for $r\in(0,1)$, 
$$
0\leq \Phi(r)\leq -(n-s)\ln r\quad \Longrightarrow \quad 0\leq e^{\Phi(r)}\leq r^{s-n}.
$$ 
Since $\rho_0$ is strictly positive in the interior, we only need to show the constant exists for $r\in[\frac{1}{3}, 1)$. 
By (\ref{eq.integralF}) , 
\begin{equation*}
\begin{aligned}
0\leq \frac{1+rF(r)}{\rho_0} &=  \frac{\rho_0^{2s-n-2}}{r^{n-1} e^{\Phi(r)}} \int_0^{r}\frac{[(n-s)(1+\bar{r}F(\bar{r}))+(s-1)]\bar{r}^{n-2}e^{\Phi(\bar{r})}}{\rho_0(\bar{r})^{2s-n-1}} d\bar{r}
\\
&\leq C_1\rho_0^{2s-n-2} \int_0^{r} \frac{\bar{r}^{s-2}}{\rho_0(\bar{r})^{2s-n-1}} d\bar{r}
\\
&\leq C_1\rho_0^{2s-n-2} \int_0^{r} \frac{\bar{r}}{\rho_0(\bar{r})^{2s-n-1}} d\bar{r}
\\
&\leq \frac{C_1 }{2s-n-2}.
\end{aligned}
\end{equation*}
Here $C_1>0$ is independent of $s\in (\frac{n}{2}+1, n)$ and $r\in[\frac{1}{3}, 1)$.  
\end{proof}
\begin{lemma}\label{lem.3.4}
For $s=\frac{n}{2}+\gamma\in (\frac{n+3}{2}, n)$ and $r\in[\frac{1}{3}, 1)$, there exists a constant $C>0$ independent of $s$ and $r$   such that 
$$
\begin{gathered}
\left| \frac{1+rF}{\rho_0}-\frac{s-1}{2s-n-2}\right| \leq \frac{C\rho_0}{2s-n-3}. 
\end{gathered}
$$
\end{lemma}
\begin{proof}
Again from (\ref{eq.integralF}) and integrating by parts,  we get
\begin{equation}\label{eq.integralF2}
\begin{aligned}
\frac{1+rF(r)}{\rho_0} &=  \frac{\rho_0^{2s-n-2}}{r^{n-1} e^{\Phi(r)}}  \int_0^{r} \frac{[(n-s)(1+\bar{r}F(\bar{r}))+(s-1)]\bar{r}^{n-3}e^{\Phi(\bar{r})}} 
{2s-n-2} d\left(\frac{1}{\rho_0(\bar{r})^{2s-n-2}}\right)
=I - II; 
\\
I &= \frac{(n-s)(1+rF)+(s-1)}{(2s-n-2)r^2} , 
\\
II &= \frac{\rho_0^{2s-n-2}}{(2s-n-2)r^{n-1} e^{\Phi(r)}}     \int_0^{r} \frac{d[(n-s)(1+\bar{r}F(\bar{r}))+(s-1)]\bar{r}^{n-3}e^{\Phi(\bar{r})}}{\rho_0(\bar{r})^{2s-n-2}}
\\
&= \frac{\rho_0^{2s-n-2}}{(2s-n-2)r^{n-1} e^{\Phi(r)}}    \int_0^r \frac{A(\bar{r})\bar{r}^{n-4}e^{\Phi(\bar{r})}}{\rho_0(\bar{r})^{2s-n-2}} d\bar{r}, 
\\
A(r)&=(n-s)(r^2F'+rF)+[(n-s)(1+rF)+(s-1)] [(n-3)+(n-s)rF].
\end{aligned}
\end{equation}
For $I$, by Lemma \ref{lem.3.3} it is obviously that for $r\in[\frac{1}{3}, 1)$, 
$$
\left| I-\frac{s-1}{2s-n-2}  \right| \leq \left(\frac{(n-s)C_1}{2s-n-2}+\frac{2(s-1)}{(2s-n-2)r^2}\right)\rho_0 \leq \frac{C_2\rho_0}{2s-n-2}. 
$$
Here  $C_1>0, C_2>0$ are constants independent of $r\in[\frac{1}{3}, 1)$ and $s\in (\frac{n+3}{2}, n)$. 
For $II$, using equation (\ref{eq.F}), Lemma \ref{lem.3.2} and Lemma \ref{lem.3.3}, we get for all $r\in [0,1)$,
$$
\begin{aligned}
\left|r^2F'\right|=\left|(n-s)(rF)^2 + (2s-n-1)r^2\left(\frac{1+rF}{\rho_0}\right) +nrF\right| \leq C_3.
\end{aligned}
$$
Hence $|A(r)|\leq C_4$. Here $C_3>0, C_4>0$ are constants independent of $r\in(0,1)$ and $s\in (\frac{n+3}{2}, n)$. Thus for $r\in[\frac{1}{3}, 1)$,
$$
\begin{aligned}
|II | & \leq C_5 \rho_0^{2s-n-2} \int_0^r \frac{\bar{r}^{s-4}}{\rho_0(\bar{r})^{2s-n-2}} d\bar{r}
\\
& \leq C_6\rho_0^{2s-n-2}\left( \int_0^{1/3}\bar{r}^{s-4} d\bar{r} +  \int_{1/3}^r \frac{\bar{r}}{\rho_0(\bar{r})^{2s-n-2}} d\bar{r}\right)
\\
& \leq \frac{C_7}{2s-n-3}\rho_0. 
\end{aligned}
$$
Here $C_5>0, C_6>0, C_7>0$ are constants independent of $r\in[\frac{1}{3},1)$ and $s\in (\frac{n+3}{2}, n)$. 
\end{proof}

By  Lemma \ref{lem.3.2}-\ref{lem.3.4} and equation (\ref{eq.T}), some direct computations show that 
\begin{lemma}\label{lem.3.5}
For $s=\frac{n}{2}+\gamma\in (\frac{n+3}{2}, n)$, $n\geq 4$ and $r\in[\frac{1}{3}, 1)$, there exists a constant $C>0$ independent of $s$ and $r$   such that 
$$
\begin{gathered}
\left |T''+\frac{n+1-s}{2s-n-2}\right| \leq \frac{C\rho_0}{2s-n-3}, 
\quad
|1+T'| \leq C\rho_0, 
\quad
|T| \leq C\rho_0,
\quad
|t-1| \leq C\rho_0. 
\end{gathered}
$$
\end{lemma}

\subsection{Curvature estimates of adapted metric $g_*$.}
We consider the weighted scalar curvature and weighted Schouten tensor for adapted metric $g_*=\rho_*^2 g_+$. Recall the formulae  given in Lemma 3.2 of \cite{CC1}:
$$
\begin{gathered}
J_{\phi_1}^{m_1} =- \rho_*^{-2} \left( |\nabla \rho_*|^2_{g_*} -1\right),
\\
P_{\phi_1}^{m_1}= P_{g_*}=\frac{1}{n-1}\left(Ric_{g_*}-\frac{R_{g_*}}{2n}g_*\right). 
\end{gathered}
$$
\begin{lemma}\label{lem.3.6}
On $\mathbb{H}^{n+1}$, 
for $s=\frac{n}{2}+\gamma \in (\frac{n+3}{2},n), n\geq 4$ and $r\in[\frac{1}{3}, 1)$, 
there exists a constant $C>0$ independent of $s$ and $r$ such that
$$
\left| J_{\phi_1}^{m_1} - \frac{n}{2s-n-2} \right| \leq \frac{C\rho_0}{2s-n-3}. 
$$
\end{lemma}
\begin{proof}
Here $\rho_*=\rho_0e^T$ and
$$
J_{\phi_1}^{m_1}= e^{-2T}\left(\frac{2(1+rT')}{\rho_0}-|T'|^2\right). 
 $$
 Then applying the estimates given by Lemma \ref{lem.3.2}-\ref{lem.3.5} we get the estimate. 
\end{proof}

To deal with the weighted Schouten tensor for adapted metric $g_*$, we first take the polar coordinates $(r,\theta)$ for $r\in[\frac{1}{3},1)$. At every point $\theta\in\mathbb{S}^n$, take $\theta=(\theta_1, ..., \theta_n) $ to be the normal coordinates on the sphere. Then the conformal relation $g_*=e^{2T}g_0$ together with Lemma \ref{lem.3.2}-\ref{lem.3.5} gives the estimates of each component of $Ric_{g_*}$ and hence $P_{g_*}$.

\begin{lemma}\label{lem.3.7}
On $\mathbb{H}^{n+1}$
for $s=\frac{n}{2}+\gamma\in (\frac{n+3}{2},n), n\geq 4$ and $r\in[\frac{1}{3}, 1)$  there exists a constant $C>0$ independent of $s$ and $r$  such that
$$
\begin{aligned}
\left|[P_{g_*}]_{rr}-\frac{n}{2(2s-n-2)}\right| \leq  \frac{C\rho_0}{2s-n-3}, 
\quad
\left|[P_{g_*}]_{\theta_i\theta_j}-\frac{1}{2}\delta_{ij}\right| \leq  \frac{C\rho_0}{2s-n-3},
\quad
[P_{g_*}]_{r\theta_j}=0. 
\end{aligned}
$$
\end{lemma}
\begin{proof}
In polar coordinates $(r,\theta)$ where $\theta=(\theta_1, ..., \theta_n)$ are normal coordinates on the sphere at some point $p$, $g_0=dr^2+r^2d\theta^2$. Let $\bar{\nabla}$ be the covariant derivative w.r.t. $g_0$.  Then at $p$, 
$$
\bar{\nabla}_r^2 T=T'', 
\quad
\bar{\nabla}_r \bar{\nabla}_{\theta_i} T=0, 
\quad
\bar{\nabla}_{\theta_i} \bar{\nabla}_{\theta_j}T= rT' \delta_{ij}, 
\quad
\triangle_{g_0}T=-T''-\frac{n}{r}T'.
$$
Since $Ric_{g_0}=0$, we have
$$
\begin{aligned}
\quad
&[Ric_{g_*}]_{rr} = -(n-1)(T''-|T'|^2) + (\triangle_{g_0}T-(n-1)|T'|^2), 
\\
&[Ric_{g_*}]_{\theta_i\theta_j} = [-(n-1)rT'+ (\triangle_{g_0}T-(n-1)|T'|^2)r^2]\delta_{ij}, 
\\
&[Ric_{g_*}]_{r\theta_j} =0. 
\end{aligned}
$$
Applying the estimates of $T$ in Lemma \ref{lem.3.2}-\ref{lem.3.5}, 
$$
\begin{aligned}
&\left|[Ric_{g_*}]_{rr}- \frac{n(s-1)}{2s-n-2}\right| \leq  \frac{C\rho_0}{2s-n-3}, 
\\
&\left|[Ric_{g_*}]_{\theta_i\theta_j}-\left(n+\frac{n+1-s}{2s-n-2}\right)\delta_{ij}\right| \leq  \frac{C\rho_0}{2s-n-3},
\\
&\left|[Ric_{g_*}]_{r\theta_j}\right|=0. 
\end{aligned}
$$
Hence the scalar curvature of $g_*$ satisfies
$$
\left|R_{g_*} - \frac{n(2s-n-1)}{2s-n-2}\right| \leq  \frac{C\rho_0}{2s-n-3}. 
$$
Since
$P_{g_*}=\frac{1}{n-1}(Ric_{g_*}-\frac{1}{2n}R_{g_*}g_*)$, we finish the proof. 
\end{proof}
%

\vspace{0.2in}
\section{Inequalities on $\mathbb{S}^4$}
Consider $\mathbb{S}^4$ as the conformal infinity of $\mathbb{H}^{4+1}$ and take the boundary metric to be the standard spherical metric $d\theta^2$. 
For $\gamma\in (0,2)$, the sharp fractional power Sobolev inequality on $\mathbb{S}^4$ is
\begin{equation}\label{ineq.s}
Y_{\gamma}(\mathbb{S}^4)   \|f\|^2_{L^{\frac{4}{2-\gamma}}(\mathbb{S}^4)}
\leq \int_{\mathbb{S}^4} fP^{}_{2\gamma} fd\theta, \quad \forall \ f\in H^{\gamma}(\mathbb{S}^4), 
\end{equation}
where
\begin{equation}\label{eq.5.1}
Y_{\gamma}(\mathbb{S}^4)=\frac{n-2\gamma}{2}Q_{2\gamma}|\mathbb{S}^4|^\frac{2\gamma}{n}=\frac{\Gamma(2+\gamma)}{\Gamma(2-\gamma)}|\mathbb{S}^4|^\frac{\gamma}{2}. 
\end{equation}
The generalised Moser-Trudinger-Onofri inequality on $\mathbb{S}^4$ was given by the energy associated to the Paneitz operator $P_4$ on it: 
\begin{equation}\label{ineq.mto}
3 \ln \fint e^{4\omega}
\leq \fint_{\mathbb{S}^4}\left(|\triangle_{\theta}\omega|^2+2|\nabla_{\theta}\omega|^2 \right)d\theta + 3 \fint 4\omega, \quad \forall \ \omega\in H^{\gamma}(\mathbb{S}^4). 
\end{equation}
This inequality was first established in \cite{BCY1} and \cite{Be1}, and extended to general higher dimensions in \cite{Be1}. 

\begin{theorem}\label{thm.2}
We can derive the generalised Morser-Trudinger-Onofri inequality (\ref{ineq.mto}) as the limit of Sobolev inequality (\ref{ineq.s}). 
\end{theorem}
\begin{proof}

On $\mathbb{S}^4$, for any $\omega\in C^{\infty}(\mathbb{S}^4)$,  take $f=e^{(2-\gamma)\omega}$. Apply 
 the sharp Sobolev inequality (\ref{ineq.s}) to $f$ and rewrite the inequality as
\begin{equation}\label{eq.5.2}
Y_{\gamma}(\mathbb{S}^4) \left( \int_{\mathbb{S}^4} |f|^{\frac{4}{2-\gamma}} d\theta \right)^{\frac{2-\gamma}{2}} 
-(2-\gamma) \int_{\mathbb{S}^4}Q_{2\gamma} |f|^2 d\theta
\leq \int_{\mathbb{S}^4} f P_{2\gamma}f d\theta - (2-\gamma) \int_{\mathbb{S}^4} Q_{2\gamma}|f|^2 d\theta. 
\end{equation}
Divide both side of (\ref{eq.5.2}) by 
$(2-\gamma)^2$. Then by taking $\gamma\rightarrow 2$, the LHS of (\ref{eq.5.2}) tends to $3|\mathbb{S}^4|$ times 
$\ln \fint e^{4\omega-4\bar{\omega}}d\theta$, where $\bar{\omega}$ is the average of $\omega$ over $\mathbb{S}^4$, and the RHS of  (\ref{eq.5.2}) tends to $|\mathbb{S}^4|$ times the high energy $ \fint_{\mathbb{S}^4}\left(|\triangle_{\theta}\omega|^2+2|\nabla_{\theta}\omega|^2 \right)d\theta$. This establishes the theorem. 

To prove this, we restrict $\gamma\in (1,2)$ and set $m_1=3-2\gamma\in (-1,1)$. Let $\rho_*$ be the adapted boundary defining function and $g_*=\rho_*^2g_+$. Let $U_f$ be defined in (\ref{eq.2.5}). Then  (\ref{eq.2.4}) (\ref{eq.5.1}) (\ref{eq.5.2}) and Lemma \ref{lem.2.2} imply that
\begin{equation}\label{eq.5.3}
\begin{aligned}
&\frac{\Gamma(2+\gamma)}{\Gamma(2-\gamma)} \left[ |\mathbb{S}^4|^{\frac{\gamma}{2}}  \left( \int_{\mathbb{S}^4} |f|^{\frac{4}{2-\gamma}} d\theta \right)^{\frac{2-\gamma}{2}}
- \int_{\mathbb{S}^4} |f|^2 d\theta  \right]
\\
\leq\  &\int_{\mathbb{S}^4}\left[ fP_{2\gamma}f - (2-\gamma)Q_{2\gamma}f^2 \right] d\theta
\\
=\ &  
\frac{d_{\gamma}}{8\gamma(\gamma-1)}
\int_{\mathbb{B}^5}  \left(  |\triangle_{\phi_1} U_f |^2+(n+m_1-1)J_{\phi_1}^{m_1}|\nabla U_f|^2_{g_*}
-4P_{ij} \nabla^iU_f \nabla^jU_f \right) \rho_*^{m_1} \mathrm{dvol}_{g_*}
\\
\leq\  & \frac{d_{\gamma}}{8\gamma(\gamma-1)}
\int_{\mathbb{B}^5}  \left(  |\triangle_{\phi_1} V|^2+(n+m_1-1)J_{\phi_1}^{m_1}|\nabla V|^2_{g_*}
-4P_{ij}\nabla^iV \nabla^jV \right) \rho_*^{m_1} \mathrm{dvol}_{g_*}, 
\quad \forall V\in \mathcal{V}^1_f. 
\end{aligned}
\end{equation}
Here $\nabla, \triangle_{\phi_1}, J_{\phi_1}^{m_1}, P$ are all with respect to the adapted metric $g_*$.

Extend $\omega$ to a smooth function on the ball:
$$
\Omega(r,\theta)=\chi(r)\omega(\theta)
$$
where $\chi\in C^{\infty}([0,1])$ satisfies $0\leq \chi\leq 1$, $\chi(r)=0$ for $r\in [0,\frac{1}{3})$ and $\chi(r)=1$ for $\gamma\in (\frac{1}{3}, 1]$. 
Define
$$
V=e^{(2-\gamma)\Omega}.
$$
Obviously, $V|_M= f$ and $V\in\mathcal{V}^1_f$. Then (\ref{eq.5.3}) implies
\begin{equation}\label{eq.5.4}
\begin{aligned}
&\frac{\Gamma(2+\gamma)}{\Gamma(2-\gamma)} \left[ |\mathbb{S}^4|^{\frac{\gamma}{2}}  \left( \int_{\mathbb{S}^4} e^{4\omega} d\theta \right)^{\frac{2-\gamma}{2}}
- \int_{\mathbb{S}^4} e^{2(2-\gamma)\omega} d\theta  \right]
\\
\leq\ & \frac{d_{\gamma}}{8\gamma(\gamma-1)} \left(2-\gamma\right)^2
\int_{\mathbb{B}^5} e^{2(2-\gamma)\Omega} \left(  \left[\triangle_{\phi_1} \Omega-(2-\gamma) |\nabla \Omega|^2_{g_*}\right]^2 \right.
\\
&\quad\quad\quad\quad\quad\quad\quad\quad\quad\quad\quad
\left.+(n+m_1-1)J_{\phi_1}^{m_1}|\nabla \Omega|^2_{g_*}
-4P_{ij}  \nabla^i\Omega \nabla^j\Omega \right) \rho_*^{m_1} \mathrm{dvol}_{g_*}
\end{aligned}
\end{equation}
For $\gamma\in (1,2)$, divide both side of (\ref{eq.5.4}) by $(2-\gamma)^2$ and we get 
\begin{equation}\label{eq.5.5}
\begin{aligned}
A_1(\gamma,\omega)=&\ 
\frac{\Gamma(2+\gamma)}{2\Gamma(3-\gamma)}\frac{2}{(2-\gamma)} 
\left[ |\mathbb{S}^4|^{\frac{\gamma}{2}}  \left( \int_{\mathbb{S}^4} e^{4\omega} d\theta \right)^{\frac{2-\gamma}{2}}
- \int_{\mathbb{S}^4} e^{2(2-\gamma)\omega} d\theta  \right]
\\
\leq\  & \frac{2^{2\gamma-3}\Gamma(\gamma)}{\Gamma(3-\gamma)}(2-\gamma)
\int_{\mathbb{B}^5} e^{2(2-\gamma)\Omega} \left(  \left[\triangle_{\phi_1} \Omega - (2-\gamma)|\nabla \Omega|^2_{g_*}\right]^2 \right.
\\
&\quad\quad\quad\quad\quad\quad\quad\quad\quad
\left.+(n+m_1-1)J_{\phi_1}^{m_1}|\nabla \Omega|^2_{g_*}
-4P_{ij}\nabla^i\Omega\nabla^j\Omega \right) \rho_*^{m_1} \mathrm{dvol}_{g_*}
\\
=\ &B_1(\gamma, \omega).
\end{aligned}
\end{equation}
Then the theorem is proved by Lemma \ref{lem.5.1} and Lemma \ref{lem.5.2}
\end{proof}

\begin{lemma}\label{lem.5.1}
For $\gamma\in (1,2)$ and  $A_1(\gamma,\omega)$ defined in (\ref{eq.5.5}), 
$$
\lim_{\gamma\rightarrow 2-} A_1(\gamma, \omega) =3 |\mathbb{S}^4| \ln  \left(\frac{1}{|\mathbb{S}^4|} \int_{\mathbb{S}^4} e^{4\omega-4\bar{\omega}}d\theta\right). 
$$
\end{lemma}
\begin{proof}
The proof is essentially the same as the proof of Lemma \ref{lem.4.1}. 
\end{proof}

\begin{lemma}\label{lem.5.2}
For $\gamma\in (1,2)$ and  $B_1(\gamma,\omega)$ defined in (\ref{eq.5.5}), 
$$
\lim_{\gamma\rightarrow 2-} B_1(\gamma, \omega) = \int_{\mathbb{S}^4} \left( |\triangle_{\theta} \omega|^2 + 2 |\nabla_{\theta}\omega|^2 \right) d\theta. 
$$
\end{lemma}
\begin{proof}
First by the definition of $\Omega$, we notice that $\triangle_{\phi_1}\Omega=0$ and $\nabla \Omega=0$ for $r\in [0,\frac{1}{3})$. Hence we only need to take the integration in $B_1(\gamma, \Omega)$ over $r\in[\frac{1}{3}, 1)$. Recall that $t=\rho_*/\rho_0$ and $T=\ln t$. Then by the conformal transformation $g_*=e^{2T}g_0$, we get
$$
\begin{aligned}
\triangle_{\phi_1}\Omega &= e^{-2T}\left(\triangle_{g_0} \Omega-m_1 \langle \bar{\nabla} \ln \rho_0,\bar{\nabla} \Omega\rangle_{g_0}-(n+m_1-1)\langle \bar{\nabla} T,\bar{\nabla} \Omega\rangle_{g_0}\right)
\\
&= e^{-2T}\left(\bar{\triangle}_{\phi_1} \Omega-(n+m_1-1)\langle \bar{\nabla} T,\bar{\nabla} \Omega\rangle_{g_0}\right);
\\
|\nabla \Omega|^2_{g_*} &= e^{-2T}|\bar{\nabla} \Omega|^2_{g_0}. 
\end{aligned}
$$
Here $\bar{\nabla}$ and $\bar{\triangle}_{\phi_1}$ are the connection and the weighted Laplacian  w.r.t. flat metric $g_0$. By Lemma \ref{lem.3.1}-\ref{lem.3.5} and the smoothness of $\Omega$, there exists a constant $C_1>0$ independent of $\gamma\in[\frac{7}{4},2)$ and $r\in [\frac{2}{3},1]$ such that 
$$
|\triangle_{\phi_1}\Omega|\leq C_1, \quad |\nabla \Omega|^2_{g_*}  \leq C_1. 
$$
Hence considering the first part: 
$$
\begin{aligned}
I(\gamma,\Omega)=\ &2(2-\gamma)
\int_{\mathbb{B}^5} e^{2(2-\gamma)\Omega}   \left[\triangle_{\phi_1} \Omega-(2-\gamma)|\nabla \Omega|^2_{g_*}\right]^2
 \rho_*^{m_1} \mathrm{dvol}_{g_*}
 \\
 =\ &    2(2-\gamma)
\int_{\mathbb{B}^5} e^{2(2-\gamma)\Omega} e^{2(2-\gamma)T}   \left[ \bar{\triangle}_{\phi_1} \Omega
 -(n+m_1-1)\langle \bar{\nabla} T,\bar{\nabla} \Omega\rangle_{g_0}
-(2-\gamma) |\bar{\nabla} \Omega|^2_{g_0}\right]^2  \rho_0^{3-2\gamma} \mathrm{dvol}_{g_0}
\\
=\ & 2(2-\gamma)\left(\int_{\mathbb{S}^4}\int_{\frac{1}{3}}^{\frac{2}{3}} +\int_{\mathbb{S}^4}\int_{\frac{2}{3}}^1 \right)
\\
=\ & I_1(\gamma, \Omega)+ I_2(\gamma, \Omega). 
\end{aligned}
$$
From the uniform bound of integrand, it is obviously that
$$
\lim_{\gamma\rightarrow 2-}I_1(\gamma, \Omega)=0.
$$ 
While $r\in[\frac{2}{3}, 1]$, 
$\Omega(r,\theta)=\omega(\theta), \partial_r\Omega=0,$
which imply that
$$
\triangle_{g_0}\Omega= \frac{1}{r^2} \triangle_{\theta}\omega, \quad 
\langle \bar{\nabla} T, \bar{\nabla} \Omega\rangle_{g_0}=0, \quad
\langle  \bar{\nabla} \ln \rho_0, \bar{\nabla} \Omega\rangle_{g_0}=0, \quad
| \bar{\nabla} \Omega|^2_{g_0}=\frac{1}{r^2}|\nabla_{\theta}\omega|^2. 
$$
Therefore, similar as the proof of Lemma \ref{lem.4.2}, we have
$$
\begin{aligned}
\lim_{\gamma\rightarrow 2-}I_2(\gamma, \Omega) 
= \lim_{\gamma\rightarrow 2-} 2(2-\gamma) \left(\int_{\frac{2}{3}}^1  e^{2(2-\gamma)T} \rho_0^{3-2\gamma} dr\right)
\left( \int_{\mathbb{S}^4}  e^{2(2-\gamma)\omega}  |\triangle_{\theta}\omega|^2 d\theta\right)
=  \int_{\mathbb{S}^4}   |\triangle_{\theta}\omega|^2 d\theta. 
 \end{aligned}
$$
Next, consider the second part:
$$
\begin{aligned}
II(\gamma, \Omega)=\ & 2(2-\gamma)
\int_{\mathbb{B}^5} e^{2(2-\gamma)\Omega} \left[(n+m_1-1)J_{\phi_1}^{m_1}|\nabla \Omega|^2_{g_*}
-4P_{ij}\nabla^i\Omega\nabla^j\Omega \right] \rho_*^{m_1} \mathrm{dvol}_{g_*}
\\
=\  & 2(2-\gamma)
\int_{\mathbb{B}^5} e^{2(2-\gamma)\Omega} e^{2(2-\gamma)T} 
\left[(n+m_1-1)e^{2T}J_{\phi_1}^{m_1}| \bar{\nabla} \Omega|^2_{g_0}
-4P_{ij} \bar{\nabla}^i\Omega   \bar{\nabla}^j\Omega \right] \rho_0^{m_1} \mathrm{dvol}_{g_0}
\\
=\ & 2(2-\gamma)\left(\int_{\mathbb{S}^4}\int_{\frac{1}{3}}^{\frac{2}{3}} +\int_{\mathbb{S}^4}\int_{\frac{2}{3}}^1 \right)
\\
=\ & II_1(\gamma, \Omega)+ II_2(\gamma, \Omega). 
\end{aligned}
$$
For $II_1(\gamma, \Omega)$, Lemma \ref{lem.3.1}-\ref{lem.3.7} and the smoothness of $\Omega$ give the uniform bound of the integrand for $r\in[\frac{1}{3}, \frac{2}{3}]$, which implies that 
$$
\lim_{\gamma\rightarrow 2-} II_1(\gamma, \Omega)=0. 
$$
For $II_2(\gamma, \Omega)$, Lemma \ref{lem.3.6}-\ref{lem.3.7} shows that there exists a constant $C_2>0$ independent of $\gamma\in[\frac{7}{4},2)$ and $r\in [\frac{2}{3},1]$ such that
$$
\begin{gathered}
 \left((n+m_1-1)e^{2T}J_{\phi_1}^{m_1}| \bar{\nabla} \Omega|^2_{g_0}
-4P_{ij} \bar{\nabla}^i\Omega   \bar{\nabla}^j\Omega \right) - \left(\frac{8(2-\gamma)}{\gamma-1}+2\right)|\nabla_{\theta}\omega|^2 =E(\gamma, \Omega)\rho_0;
\\
\left| E(\gamma,\Omega)\right| \leq C_2. 
\end{gathered}
$$
Thus
$$
\begin{aligned}
\lim_{\gamma\rightarrow 2-} II_2(\gamma, \Omega)
=&\ \lim_{\gamma\rightarrow 2-} 2(2-\gamma)
\int_{\mathbb{S}^4}\int_{\frac{2}{3}}^1
e^{2(2-\gamma)\omega} e^{2(2-\gamma)T} \left(\frac{8(2-\gamma)}{\gamma-1}+2\right)|\nabla_{\theta}\omega|^2 \rho_0^{3-2\gamma}
r^4drd\theta
\\
=\ & \int_{\mathbb{S}^4} 2|\nabla_{\theta}\omega|^2 d\theta.
\end{aligned}
$$
We finish the proof.
\end{proof}

\begin{remark}
It turns out when $n$ is even, there is a generalisation of the Moser-Trudinger-Onofri inequality (\cite{Be1}) for functions defined on $(\mathbb{S}^n,d\theta^2)$ with the role of Laplace operator on $(\mathbb{S}^2,d\theta^2)$ replaced by the $n$-th order GJMS operator. It is plausible the argument in this paper can be applied to derive these inequalities as the limit of Sobolev embedding inequalities for the corresponding fractional $2\gamma$-order GJMS operators as $\gamma\rightarrow n/2$. But as the proof of our theorems (for $n=2$ and $n=4$ cases) indicate, the argument would become increasingly delicate when $n$ is large. 
\end{remark}

\vspace{0.2in}

\end{document}